\documentclass [11pt]{article}
\usepackage{a4}
\usepackage{paralist}
\usepackage{amssymb}
\usepackage{float}
\usepackage{amsmath}
\usepackage{psfrag}
\usepackage{amsmath,amsthm}
\usepackage{cleveref}

\usepackage[all]{xy}

\usepackage[pdftex]{graphicx}

\xyoption{web}

\usepackage{amsmath,amsthm}

\theoremstyle{plain}% default
\newtheorem{thm}{Theorem}[section]

\newtheorem{lem}[thm]{Lemma}
\newtheorem{prop}[thm]{Proposition}
\newtheorem{cor}[thm]{Corollary}

\theoremstyle{definition}
\newtheorem{defn}[thm]{Definition}
\newtheorem{ex}[thm]{Example}%[section]
\theoremstyle{remark}
\newtheorem{rem}[thm]{Remark}

\title{\textbf{Presentations with negative deficiency and largeness}}
\author{Mariano Zer\'{o}n-Medina Laris}

\begin{document}

\maketitle

\begin{abstract}
We extend the results on balanced presentations and largeness in \cite{Edjvet-balanced} by giving conditions under which presentations with negative deficiency define large groups.
\end{abstract}

\section{Introduction}

A finitely generated group $G$ is large when it has a finite index subgroup that surjects 
onto a non-abelian free group of rank $2$ (\cite{pride large}). This is a strong property 
preserved under finite index subgroups, finite supergroups and pre-quotients. Moreover, 
if a group is large then it contains a non-abelian free subgroup \cite{neumann}, it is 
$SQ$-universal \cite{pride large} (every countable group is a subgroup of a quotient of $G$), 
it has subgroups with arbitrarily large first Betti number \cite{lubotzky}, uniformly 
exponential word growth \cite{harpe}, as well as subgroup growth of strict type $n^{n}$ 
\cite{lubotzky segal}, among other properties.
\vskip 2mm

The deficiency of a finite presentation is defined as the number of generators minus the 
number of relators. The deficiency of a finitely presented group $G$, is defined as 
the supremum over the deficiencies of the finite presentations which define the group $P$.
\vskip 2mm

The deficiency of a group is a numerical invariant linked to largeness. The first result to show this was proved in $1978$ by B. Baumslag and S. Pride proved (\cite{baums-pride}). The theorem says that if a presentation has deficiency greater than one, then it defines a large group. Since then many of the 
results on largeness use such theorem.
\vskip 2mm

Shortly after the publication of \cite{baums-pride}, the question of largeness in 
finitely presented groups with lower deficiency was explored. In \cite{baums-pride 2}, 
presentations with deficiency one and a proper power relator were studied. A condition was 
found under which such a presentation yields a group with a finite index subgroup with 
deficiency greater than one. By \cite{baums-pride}, this means the group is large. This 
condition, however, does not apply to all deficiency one presentations with a proper power 
relator. They nevertheless conjectured that a deficiency one presentation with a proper power 
relator should always define a large group. The conjecture was settled first by 
M. Gromov in \cite{gromov} by considering bounded cohomology and then by R. Stohr 
in \cite{Stohr} using direct algebraic methods. However, the question of whether a 
deficiency one presentation with a proper power relator defines a group which has a 
finite index subgroup with deficiency greater than one, is still open.
\vskip 2mm

Not long after, M. Edjvet proved (\cite{Edjvet-balanced}) that under certain conditions, 
groups defined by balanced presentations (i.e. equal number of generators and relators)
are also large. His proof is divided into two cases. The first deals with groups that 
have finite abelianisation. He proves that if the presentation has at least two proper power relators where one of them has a power greater than two, then the commutator subgroup has deficiency greater than one. The second case considers groups with infinite abelianisation. Here he shows that R. Stohr's main result (\cite{Stohr}) applies and therefore concludes largeness if at least two of the powers have a non-trivial common factor.
\vskip 2mm

In this paper we give conditions under which a presentation with deficiency minus one defines a large group. We follow a similar approach to the one in \cite{Edjvet-balanced} dividing the proof into the same two cases. In the first case, we show, by reducing it down to the deficiency zero case, that under suitable conditions, the commutator of a group given by a deficiency minus one presentation has deficiency greater than one. When the group has infinite abelianisation, we prove that R. Stohr's result can still be used under suitable conditions on the powers of the presentation. The latter result applies to presentations of negative deficiency and not only to those with deficiency minus one.
\vskip 2mm

\textbf{Acknowledgements}
\vskip 2mm

We would like to thank M. Edjvet for suggesting the proof of \cref{lemma for thm}. This work is part of the author's PhD done under the supervision of Jack Button. The author would like to thank Jack for all the support and advice needed for the completion of this work. Also, the author would like to thank the Mexican Council of Science and Technology (CONACyT) and the Cambridge Overseas Trust for their financial support all these years.

\section{Deficiency minus one presentations with finite non-\\trivial abelianisation}

\vskip 2mm

We will use the following notion introduced in \cite{yo}.
\vskip 2mm

\begin{defn}\label{def: residual def}
Let $G$ be a finitely presented group with finite presentation $Q=\langle X|R\rangle$, 
where $X$ freely generates $F_n$, the non-abelian group of rank $n$. 
Let $R=\{u^{s_1}_1,\ldots,u^{s_m}_m\}$, where $u_i$ is the minimal root 
of $u_i^{s_i}$ for all $i$, $1\leq i\leq m$. Suppose the order of $\psi(u_i)$ in the 
residual quotient of $G$ is $k_i$, for all $i$. Then we define the 
\emph{residual deficiency} of the presentation $Q$ to be
\[
rdef(Q)=n-\sum_{i=1}^{m}\dfrac{1}{k_i}.
\]
We define the \emph{residual deficiency} of the group $G$ to be the
supremum  of the residual deficiencies defined by all finite presentations of $G$
\[
rdef(G)=\underset{\langle X|R\rangle \cong G}{\text{sup}} \left\{ rdef(Q) \right\}.
\]
\end{defn}
\vskip 2mm

The residual deficiency helps give a lower bound for the deficiency of finite index subgroups. This is given by the next theorem \cite{yo}.
\vskip 2mm

\begin{thm}$($Theorem $3.4$ in \cite{yo}$)$\label{deficiency for finite index subgroups}\\
Let the group $G$ be given by the presentation
\[
P=\langle x_1,\ldots,x_n \mid u_1^{m_1},\ldots,u_r^{m_r}\rangle,
\]
where $n, m_i\geq 1$ for $i=1,\ldots,r$. Let $k_i$, for $1\leq i\leq r$, be as in \cref{def: residual def}, the order of $\varphi(u_i)$ in $G/R_G$. Then, there are finite index normal subgroups $H$ in $G$, such that the order of $\varphi(u_i)$ in $G/H$ is $k_i$, for all $i$, $1\leq i\leq r$. Moreover, the deficiency of every such $H$ is bounded below by
\[
1+|G:H|(rdef(P)-1).
\]
\end{thm}
\vskip 2mm

Given the presentation
\[
P=\langle x_1,\ldots,x_d \mid u^{m_1}_1, \ldots, u^{m_{d+1}}_{d+1}\rangle,
\]
consider the presentation $P_i$ obtained by removing the relator $u_i^{m_i}$ from $P$,
\[
P_i=\langle x_1,\ldots,x_d \mid u^{m_1}_1,\ldots,u^{m_{i-1}}_{i-1},u^{m_{i+1}}_{i+1},\ldots, u^{m_{d+1}}_{d+1}\rangle.
\]
Denote by $G_i$ the group defined by the presentation $P_i$.
\vskip 2mm

This section considers the case when $G$ has finite non-trivial abelianisation. In general, whenever we have proper power relators we assume the powers are greater than or equal to one. First we prove the following lemma, which is used in the main theorem of \cite{Edjvet-balanced}. \footnote{This lemma is not proved in \cite{Edjvet-balanced}. The proof presented here was suggested to the author by M. Edjvet through private communication.}
\vskip 2mm

\begin{lem}\label{lemma for thm}
Let $G$ be a finitely presented group given by the following presentation
\[
Q=\langle x_1,\ldots,x_d \mid u^{m_1}_1,\ldots, u^{m_d}_{d}\rangle.
\]
Let $H$ be the commutator subgroup of $G$ and assume $G$ has finite abelianisation. Denote by $\overline{H}$ the inverse image of $H$ under $\varphi$, where $\varphi$ is the canonical map $\varphi:F_d\longrightarrow G$ induced by the presentation $Q$.
%where $\varphi$ is the map from $F_d$, the non-abelian free group of rank $d$, to $G$, induced by the presentation $P$.
Then, the order of $u_i$ in $F_d/\overline{H}$, is $m_i$, for all $i$, $1\leq i\leq d$.
\end{lem}
\begin{proof}
Assume, by repeated application of Lemma 11.8 in \cite{lyndon} (p. 293), that the presentation above satisfies $\sigma_{x_j}(u_i)=0$ if $1\leq i<j\leq d$, where $\sigma_{x_j}(u_i)$ is the exponent sum of the generator $x_j$ in the word $u_i$. For the remainder of the proof, denote $\sigma_{x_j}(u_i)$ by $a_{i,j}$.
\vskip 2mm

The abelianisation of the group $G$ admits the following presentation
\[
ab(Q)=\langle x_1,\ldots,x_d \mid u^{m_1}_1,\ldots, u^{m_d}_{d}, [x_i,x_j]\rangle,
\]
where $1\leq i<j\leq d$. Moreover, by the above argument we may assume $a_{i,j}=0$ for $1\leq i<j\leq d$, hence for all $i$ such that $1\leq i\leq d$, the word $u_i^{m_i}$ may be written as $(x_1^{a_{i_1}}x_2^{a_{i_2}}\cdots x_i^{a_{i_i}})^{m_i}$.
\vskip 2mm

Denote by $M_{ab(Q)}$ the exponent sum matrix of $ab(Q)$. Therefore,  its $(i,j)$-th entry is $(M_{ab(Q)})_{i,j}=\sigma_{x_j}(u^{m_i}_i)=m_ia_{i_{j}}$. %Denote this matrix by $M_{ab(P)}$. Note that $(M_{ab(P)})_{i,j}=m_iM_{i,j}$.

\vskip 2mm

Since the presentation $Q$ is balanced, the exponent sum matrix $M_{ab(Q)}$ is square. Moreover, the absolute value of the determinant of the matrix $M_{ab(Q)}$ gives the order of the abelianisation of $G$, which is the same as the index of $H$ in $G$.
\vskip 2mm

Note that as $\sigma_{x_j}(u_i)=0$ if $0\leq i<j\leq d$, the determinant is the multiplication of the diagonal elements $\prod_{i=1}^{d}(M_{ab(Q)})_{i,i}=\prod_{i=1}^{d}m_ia_{i_{i}}$.
\vskip 2mm

For all $i$, $1\leq i\leq d$, let $\overline{u}_i$ be the element in the abelianisation of $G$ that corresponds to the word $u_i$. Assume the order of $\overline{u}_i$ is $n_i<m_i$ for some $i$, $1\leq i\leq d$. Exchange the relator $(x_1^{a_{i_1}}x_2^{a_{i_2}}\cdots x_i^{a_{i_i}})^{m_i}$ in the presentation $ab(Q)$, for $(x_1^{a_{i_1}}x_2^{a_{i_2}}\cdots x_i^{a_{i_i}})^{n_i}$. Therefore, the exponent sum matrix associated to this new presentation has in its $(i,i)$-th entry, $n_ia_{i_{i}}$ instead of $m_ia_{i_{i}}$. This means that the index of $H$ in $G$ is the product $m_1a_{1_{1}}\cdots n_ia_{i_{i}}\cdots m_na_{n_{n}}$, which is strictly less than $\prod_{i=1}^{d}m_ia_{i_{i}}$. This is a contradiction and hence the order of $\overline{u}_i$ is $m_i$ for all $i$, $1\leq i\leq d$.

\end{proof}
%\vskip 2mm
\vskip 2mm

Consider $I=\{1,\ldots, d+1\}$, and take $J$ to be the set of elements $j\in I$, such that $G_j$ has finite abelianisation. Say the cardinality of $J$ is $l$, where $l\geq 0$.
\vskip 2mm

\begin{thm}\label{thm: 1}
Let $G$ be a finitely presented group with presentation
\[
P=\langle x_1,\ldots,x_d \mid u^{m_1}_1, \ldots, u^{m_{d+1}}_{d+1}\rangle,
\]
where $d\geq 2$. Suppose $G$ has non-trivial finite abelianisation. Then
\begin{enumerate}
\item The commutator subgroup of $G$ has deficiency greater than one if
\[
d-l-\sum_{i\notin J}\dfrac{1}{m_{i}}>1.
\]
\item If for some $j\in J$, the image of $u_j^{m_j}$ in $G_j$ is contained in the commutator subgroup of $G_j$, then the commutator has deficiency greater than one if
\[
d-\sum_{i\in I,i\neq j}\dfrac{1}{m_{i}}-\dfrac{1}{k}>1,
\]
where $k$ is the order of $u_j$ in the abelianisation of $G_j$.
\end{enumerate}
\end{thm}
\begin{proof}

$1$) For clarity we work in $F_d$, the non-abelian free group of rank $d$. Diagram (\ref{eq: diagram subgps}) is a diagram of subgroups that is useful for the proof. Denote $\langle\langle u^{m_1}_1, \ldots, u^{m_{d+1}}_{d+1} \rangle\rangle$ by $N$ and $\langle\langle u^{m_1}_1,\ldots,u^{m_{i-1}}_{i-1},u^{m_{i+1}}_{i+1},\ldots, u^{m_{d+1}}_{d+1} \rangle\rangle$ by $N_i$. The commutator subgroup of $F_d$ will be denoted by $F'_d$.
\vskip 2mm

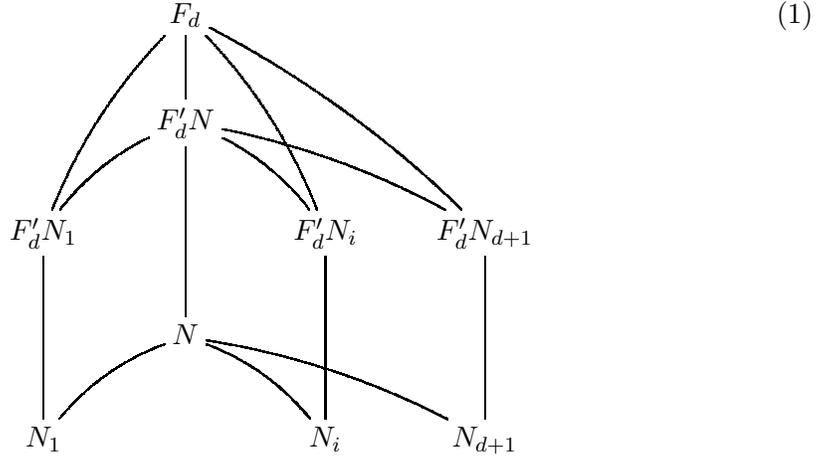
\begin{equation}\label{eq: diagram subgps}
\xymatrix{
& F_d\ar@{-}[d]\ar@{-}@/_/[ddl]\ar@{-}@/^/[ddr]\ar@{-}@/^/[ddrr] &&\\
& F'_dN\ar@{-}@/_/[dl]\ar@{-}@/^/[dr]\ar@{-}[dd]\ar@{-}@/^/[drr]  && \\
F'_dN_1\ar@{-}[dd] && F'_dN_i\ar@{-}[dd] & F'_dN_{d+1}\ar@{-}[dd] \\
& N\ar@{-}@/^/[drr]\ar@{-}@/^/[dr]\ar@{-}@/_/[dl] &\\
N_1 && N_i & N_{d+1}\\
}
\end{equation}
\vskip 5mm

The commutator subgroup of $G$ and $G_i$ correspond under $\varphi$ to $F_d'N$ and $F_d'N_i$, respectively. Note that $F'_dN_i$ is contained in $F'_dN$ for all $i\in I$. Hence, as the abelianisation of $F_d/N$ is non-trivial, the abelianisation of $F_d/N_i$ is non-trivial for all $i\in I$. Also, the order of the abelianisation of $F_d/N$ is given by the greatest common divisor  of $\{k_1,\ldots,k_l\}$, where $k _j$ is the order of the abelianisation of $F_d/N_j$ for $j\in J$, i.e. the order of $F_d/F'_dN_j$ (page $185$, \cite{jacobson}). As the abelianisation of $G$ is finite by assumption, then $J$ is non-empty.
\vskip 2mm

Say $n$ is the order of $F_d/F'_dN$, the abelianisation of $G$, and $s_j$ the order of the quotient $F'_dN/F'_dN_j$, for $j\in J$. Clearly $k_j=s_jn$ for all $j\in J$, and since $n=$gcd$\{k_1,\ldots,k_l\}$, then the set of $\{s_j\}$ with $j\in J$, has no common factors.

\vskip 2mm

Consider $i\in I$ such that $i\notin J$. Using \cref{lemma for thm}, $u_i,u_i^{2},\ldots,u_i^{m_i-1}\notin F'_dN_j$ for all $j\in J$. Say the order of $u_i$ in $F_d/F_d'N$ is $q$. Then, as $u_i,u_i^{2},\ldots,u_i^{m_i-1}\notin F'_dN_j$ for all $j\in J$, the order of $u_i^{q}$ in $F_d'N/F_d'N_j$ is $m_i/q=\lambda$ for all $j\in J$. Therefore, $\lambda$ divides $s_j$ for $j\in J$. As the set $\{s_j\}$, $j\in J$, has no common factors, then $\lambda=1$, which means $q=m_i$. Therefore, $u_i,u_i^{2},\ldots,u_i^{m_i-1}\notin F'_dN$ for all $i\notin J$. The result then follows by applying \cref{deficiency for finite index subgroups} to $G$ and its commutator subgroup using the presentation $P$.
\vskip 2mm

$2$) Since $u_j^{m_j}$ is contained in $F'_dN_j$, then $F'_dN_j=F'_dN$. As $G_j$ has finite abelianisation, \cref{lemma for thm} implies $u_i$ has order $m_i$ in $F'_dN$, for all $i\neq j$. Apply \cref{deficiency for finite index subgroups} to $G$ and its commutator subgroup using presentation $P$ to obtain the result.

\end{proof}

\vskip 2mm

\begin{rem}\label{rem: finite abel iff J non empty}
Given that $G_i$ surjects onto $G$ for all $i\in I$, if $G_i$ has finite abelianisation for some $i$, then $G$ has finite abelianisation too. From the proof of the previous theorem, if $G$ has finite abelianisation, then $J$ must be non-empty. That is, there is a $G_i$ with finite abelianisation. Therefore, $G$ has finite abelianisation if and only if $J$ is non-empty. 
\vskip 2mm

A condition that guarantees the abelianisation is non-trivial is given by the following proposition.

\end{rem}
\vskip 2mm

\begin{prop}\label{prop: Gi inf abel implies G has nontrivial abel}
Let $G$ be given by the presentation $P=\langle x_1,\ldots,x_d \mid u^{m_1}_1, \ldots, u^{m_{d+1}}_{d+1}\rangle$. Suppose $G_i$ has infinite abelianisation for some $i$ where $m_i>1$. Then $G$ has non-trivial abelianisation.
\end{prop}
\begin{proof}
By assumption $G_i$ surjects onto the integers. Denote this map by $\psi$ and denote by $\overline{u}_i$ the image of $u_i$ in $G_i$. Since $m_i>1$, then $\psi(\overline{u}_i^{m_i})=m_i\psi(\overline{u}_i)\neq 1$. Hence $G$, which is the quotient of $G_i$ by the normal subgroup generated by $\overline{u}_i^{m_i}$, surjects onto $\mathbb{Z}$ quotiented out by $\psi(\overline{u}_i^{m_i})$. As $\psi(\overline{u}_i^{m_i})\neq 1$, the latter quotient is non-trivial and hence the result follows.
\end{proof}
\vskip 2mm

\begin{ex}
Consider the group $G$ given by the presentation
\[
P=\langle x_1,\ldots,x_{2n}\mid (x_1x_2)^{m_1},x_2^{m_2},\ldots, (x_{2n-1}x_{2n})^{m_{2n-1}},x_{2n}^{m_{2n}},u^{\alpha}\rangle,
\]
where $\alpha>1$, $m_i\geq 2$, for $i=1,\ldots, 2n$, and where $u$ only depends on the generators indexed by even numbers $x_2,x_4,\ldots,x_{2n}$. Since the exponent sum matrix associated to $G_{2n+1}$ is upper triangular with non-zero entries in its diagonal, then $G_{2n+1}$ has finite abelianisation. Now consider the exponent sum matrix $M_P$ associated to $P$. Since $u$ only depends on $x_2,x_4,\ldots,x_{2n}$, and the remaining relators are $(x_1x_2)^{m_1}, x_2^{m_2},\ldots, (x_{2n-1}x_{2n})^{m_{2n-1}}, x_{2n}^{m_{2n}}$, then the odd column $2i-1$ ($1\leq i\leq n$) only has one non-zero entry: that which corresponds to the $2i-1$ relator (see \cref{matrix}). Therefore, deleting such a relator from the presentation gives a group $G_{2i-1}$ with infinite abelianisation. Hence, by \cref{prop: Gi inf abel implies G has nontrivial abel}, the group $G$ has finite non-trivial abelianisation, and by \cref{thm: 1} part $1$, its commutator has deficiency greater than one in the following cases:
\begin{itemize}
 \item $n=3$, $m_1,m_3,m_5\geq 3$ and at least one of them greater than three.
 \item $n=4$ and $m_{2i-1}> 2$ for some $i$, $1\leq i\leq 4$.
\item $n\geq 5$.
\end{itemize}
\begin{equation}\label{matrix}
M_P=\begin{pmatrix}
m_1 & m_1 & 0 & \cdots\\
0 & m_2 & 0 & \cdots\\
\vdots & 0 & m_3 & m_3 & 0 & \cdots\\
& 0 & 0 & m_4 & 0 & \cdots\\
& \vdots & \vdots & 0 & m_5 & \cdots\\
& & & \vdots \\
0 & \sigma_{x_2}(u^{\alpha}) & 0 & \sigma_{x_4}(u^{\alpha}) & 0 & \cdots\\
\end{pmatrix}
\end{equation}
\end{ex}
\vskip 2mm

\begin{ex}
Consider the group $G$ given by a balanced presentation
\[
P=\langle x_1,\ldots,x_d \mid u^{m_1}_1,\ldots, u^{m_d}_{d}\rangle,
\]
such that $G$ has finite abelianisation and where $m_i\geq 2$, for some $i$, $1\leq i\leq d$. Given that $G$ has finite abelianisation and $m_i\geq 2$ for some $i$, then by the proof of \cref{lemma for thm} $G$ has non-trivial abelianisation. Take $p$ a prime dividing the order of $ab(G)$, the abelianisation of $G$. Note that there is an element $v$ in $ab(G)$ with order equal to $p$ (Cauchy's Theorem). Take $w$ in $G$ which corresponds to $v$ under the canonical surjection from $G$ to $ab(G)$. 
Take $n$ any multiple of $p$ and consider $\overline{G}:=G/\langle\langle w^{n}\rangle\rangle$. As $v^n$ is trivial in $ab(G)$, then $G/\langle\langle w^{n}\rangle\rangle$ (and hence $\overline{G}$) has non-trivial abelianisation. Moreover, by \cref{thm: 1} part $2$, the commutator of $\overline{G}$ has deficiency greater than one if
\[
d-\sum_{i=1}^{d}\dfrac{1}{m_i}-\dfrac{1}{p}>1.
\]
The previous inequality holds for all $m_i\geq 2$ if $d\geq 4$. If $d=3$ then either $p\neq 2$ or $m_i\geq 3$ for some $i$, $1\leq i\leq d$, would be enough to ensure the inequality holds.
\end{ex}

\vskip 2mm

\cref{prop: Gi inf abel implies G has nontrivial abel} can be used to prove $G$ has non-trivial abelianisation. However, there are examples of presentations $P$ with deficiency minus one, which define a group $G$, such that $G_i$ has finite abelianisation for all $i$, $1\leq i\leq n+1$, but $G$ has non-trivial abelianisation. For these cases, the following proposition, which gives conditions under which the abelianisation of $G$ is non-trivial (provided $G$ has finite abelianisation), may be useful.
\vskip 2mm

\begin{prop}\label{prop: conditions under which G has finite non trivial ab}
Let $G$ be given by $P=\langle x_1,\ldots,x_d \mid u^{m_1}_1, \ldots, u^{m_{d+1}}_{d+1}\rangle$. %  where $m_i\geq 1$ for all $i$, $1\leq 1\leq d+1$. 
Suppose $G_i$ has finite abelianisation for some $i$. If $(m_i,m_j)\neq 1$ for some $j\neq i$,% where $\sigma_{x_j}(u_j)$ is the exponent sum of the generator $x_j$ in the word $u_j$,
 then the abelianisation of $G$ is finite and non-trivial.
\end{prop}
\begin{proof}
As the abelianisation of $G_i$ is finite and $G$ is a quotient of $G_i$, then the abelianisation of $G$ is finite too.
\vskip 2mm

Without loss of generality, assume $i=d+1$. Now consider the exponent sum matrix of $G_{d+1}$. Denote the $(j,k)$ entry of this matrix by $m_ja_{j,k}$ where $m_j$ is the power of $u_j$ and $a_{j,k}=\sigma_{x_k}(u_j)$. As in \cref{lemma for thm}, this matrix may be assumed to be lower triangular. Moreover, assume it has positive values on its diagonal so that its determinant, which is the order of the abelianisation of $G_i$, is the multiplication of the diagonal elements. Therefore, the abelianisation of $G_{d+1}$ is $\prod_{i=1}^{d}m_ia_{i,i}$.%=p_1^{k_i}\cdots p_t^{k_t}$, where the latter is the prime decomposition of the order of the abelianisation.
\vskip 2mm

Consider a surjective map from $G_{d+1}$ to the prime decomposition of the abelianisation of $G_{d+1}$
\[
G_{d+1}\longrightarrow C_{{q_1}^{l_1}}\times\cdots\times C_{{q_s}^{l_s}}.
\]
As the order of the abelianisation is $\prod_{i=1}^{s}{q_i}^{l_i}$ then $\prod_{i=1}^{s}{q_i}^{l_i}$ is the prime factorisation of $\prod_{i=1}^{d}m_ia_{i,i}$. If $(m_{d+1},m_j)\neq 1$, then there is a prime $q$ which appears in $\prod_{i=1}^{s}{q_i}^{l_i}$ and divides both $m_{d+1}$ and $m_j$. 
\vskip 2mm

Assume, without loss of generality, that $q=q_1$. Consider the projection of $C_{{q_1}^{l_1}}\times\cdots\times C_{{q_s}^{l_s}}$ onto $C_{{q_1}^{l_1}}$. Denote the composition from $G_{d+1}$ to $C_{{q_1}^{l_1}}$ by $\psi$. Then $\psi(\overline{u}_{d+1}^{m_{d+1}})=m_{d+1}\psi(\overline{u}_{d+1})$, where $\overline{u}_{d+1}$ is the image of $u_{d+1}$ in $G_{d+1}$, does not generate $C_{{q_1}^{l_1}}$ as $q_1$ divides $m_i$. Hence, $C_{{q_1}^{l_1}}$ quotiented out by $\psi(\overline{u}_{d+1}^{m_{d+1}})$ is non-trivial. Finally, $G_{d+1}$ quotiented out by $\overline{u}_{d+1}^{m_{d+1}}$ surjects onto $C_{{q_1}^{l_1}}$ quotiented out by $\psi(\overline{u}_{d+1}^{m_{d+1}})$. The result follows as the former quotient is isomorphic to $G$.
\end{proof}
\vskip 2mm

\begin{ex}
Consider $G$ the one relator quotient of a finite product of non-trivial finite cyclic groups. This group admits the following presentation
\begin{equation}%\label{eq: one relator quot of free prod of cyclics}
P=\langle x_1,\ldots,x_n \mid x_1^{m_1},\ldots,x_n^{m_n}, w^{s}\rangle.
\end{equation}
We consider examples where $m_i,s>1$, for all $i$, $1\leq i\leq n$.% and $w$ is a word in $x_1,\ldots,x_n$. 
\vskip 2mm

Now suppose $\sigma_{x_i}(w)\neq 0$ for all $i$, $1\leq i\leq n$, and suppose $s$ is a multiple of $k$, where $k=\prod_{i=1}^{n}m_i$. Then we claim $G_j$ has finite abelianisation for all $j$, $1\leq j\leq n+1$.

\vskip 2mm

This is trivial for $G_{n+1}$ as the latter is the product of finite cyclic groups. For $G_j$, $1\leq j\leq n$, consider the map $\psi$ from $F_n$, the non-abelian free group of rank $n$ freely generated by $x_1,\ldots,x_n$, to $\mathbb{Z}^{n}$, which sends $x_1$ to $(1,0,\ldots,0)$, $x_2$ to $(0,1,0\ldots,0)$, and so on. The image of $w$ in $\mathbb{Z}^{n}$ is given by $(a_1,\ldots,a_n)$, where $a_i=\sigma_{x_i}(w)$, for all $i$, $1\leq i\leq n$. By assumption, $a_i\neq 0$ for all $i$, $1\leq i\leq n$. 
\vskip 2mm

Denote $\psi(x_k)$ by $\delta_k$. We want to show that $G_j$ has finite abelianisation for $1\leq j\leq n$. For this, it suffices to show that given a fixed $j$, the order of $\delta_k$ in the quotient of $\mathbb{Z}^{n}$ over the subgroup generated by $\psi(x_i^{m_i})$ (for all $i\neq j, 1\leq i\leq n$) and $\psi(w^s)$, is finite when $1\leq k\leq n$. This is clear for $\delta_k$, where $k\neq j$, since the relator $x_k$ appears to a finite power in the presentation. So it remains to prove if for $\delta_j$.
\vskip 2mm

Note that $s$ is a multiple of $m_i$, for all $i$, $1\leq i\leq n$. Therefore, since $a_i\neq 0$ for all $i$, then $a_js\delta_j$ is in the subgroup of $\mathbb{Z}^n$ generated by $(a_1s,\ldots,a_ns)=\psi(w^s)$ and the set of $m_i\delta_i=\psi(x^{m_i})$, where $1\leq i\leq n$ but $i\neq j$.
\vskip 2mm

Finally, as $s$ is a multiple of $m_i$ for all $i$, $1\leq i\leq n$, then $G$ has non-trivial abelianisation by \cref{prop: conditions under which G has finite non trivial ab}.

\end{ex}

\section{Presentations with negative deficiency and infinite\\
abelianisation}

In \cite{Stohr}, R. Stohr proved that a deficiency one presentation with a proper power relator defines a large group. In \cite{Edjvet-balanced}, M. Edjvet noted that the proof in \cite{Stohr} also applies to groups defined by balanced presentations which define groups with infinite abelianisation, provided that the presentation has at least two proper power relators with the powers having a non-trivial common factor. His argument, however, is specific to balanced presentations and hence does not work for presentations with negative deficiency. The aim of this section is to show that Stohr's arguments extend to presentations of negative deficiency (and even to presentations with an infinite number of relators) which define groups with infinite abelianisation, as long as certain conditions, which we will present, are imposed.
\vskip 2mm

First, let us recall some of the key steps in the proof of \cite{Stohr}. Given a group $G$ with presentation $\langle x_1,\ldots,x_n\mid r_1,\ldots, r_{n-1}^{\alpha}\rangle$, where $\alpha>1$, R. Stohr uses Lemma $11.8$ in \cite{lyndon} to change it for a presentation
\[
P=\langle a_1,\ldots,a_{n-1},t\mid R_1,\ldots,R_{n-2},R_{n-1}^{\alpha}\rangle,
\]
where $\sigma_t(R_i)=0$ for all $i$, $1\leq i\leq n-1$. He notes that given a presentation such as $P$, there is a natural number $m$, such that the words $R_1,\ldots, R_{n-1}$, can be rewritten in terms of $a_{i,j}=t^{j}a_it^{-j}$, where $1\leq i\leq n-1$ and $-m\leq j\leq m$. Therefore, the group $G$ admits a presentation
\[
\tilde{P}=\langle a_{i,j},t\mid P_1,\ldots,P_{n-1}^{\alpha},ta_{i,j}t^{-1}=a_{i,j+1},\  (-m\leq j\leq m, i=1,\ldots,n-1)\rangle,
\]
where $P_i$ is $R_i$ rewritten in terms of $a_{i,j}$. Note that no $t$ appears in $P_i$, $1\leq i\leq n-1$. The remainder of the proof then concentrates on proving that a group with a presentation such as $\tilde{P}$ surjects onto the following $HNN$ extension with base $U$ an elementary abelian $p$-group of rank $2N+1$ with basis $u_{-N},\ldots,u_N$
\[
\langle U,t\mid tu_kt^{-1}=u_{k+1}, (-N\leq k<N) \rangle,
\]
where $p$ is a divisor of $\alpha$, and $N$ is a suitably chosen natural number. As such an $HNN$ extension has a finite index subgroup isomorphic to a non-abelian free group of finite rank (\cite{karras piet solitar}), then $G$ is large.
\vskip 2mm

For Stohr's argument to work, two things are key. First, the existence of a natural number $m$, such that the words $R_i$ may be rewritten in terms of $a_{i,j}$, $-m\leq j\leq m$, $i=1,\ldots,n$. Second, at most $n-2$ relators do not have a power divisible by the prime $p$ that we use to define the $HNN$ extension $H$. The relators that come to a $p$-power go to the identity under the surjection to the $HNN$ extension. The important thing to note is that the number of relators in $\tilde{P}$ is not important; as long as these two conditions are met, the group that $\tilde{P}$ defines is large.
\vskip 2mm

Now we give conditions under which such an $m$ exists. Let 
\[
\langle x_1,\ldots,x_n\mid r_1,\ldots,r_m,\ldots\rangle
\]
be a presentation that defines a group $G$ that surjects onto the integers. Consider $w=x_{i_1}^{k_{i_1}}\cdots x_{i_l}^{k_{i_l}}\in F_n$, a word in terms of $x_1,\ldots,x_n$. If $\psi:G\longrightarrow \mathbb{Z}$ is a surjective map from $G$ onto the integers, then consider $\tilde{\psi}=\psi\circ \varphi$, where $\varphi:F_n\longrightarrow G$ is the canonical map from $F_n$ to $G$. Define $\Delta_{\psi,j}(w)=k_{i_1}\tilde{\psi}(x_{i_1})+\cdots+k_{i_j}\tilde{\psi}(x_{i_j})$, where $1\leq j\leq l$. Define
\[
\Delta_{\psi}(w)=\underset{1\leq j\leq l}{\text{max}}\{ \mid\Delta_{\psi,j}(w)\mid\}.
\]
\vskip 2mm

We now consider a presentation with more relators than before,
\[
P_1=\langle a_1,\ldots,a_{n-1},t\mid R_1,\ldots,R_{n-2},R_{n-1}^{\alpha_1},\ldots,R_{n+k}^{\alpha_{k+2}} \rangle,
\]
where $\sigma_t(R_i)=0$ for all $i=1,\ldots,n+k$, and $p$ divides $\alpha_j$, for $j=1,\ldots,k+2$. As $\sigma_t(R_i)=0$ for all $i=1,\ldots,n+k$, then $G_1$, the group defined by $P_1$, surjects onto $\mathbb{Z}$ by sending $t$ to $1$ and $a_1,\ldots,a_{n-1}$ to $0$. Therefore, $\Delta_{\psi}(R_i)$ keeps track of the powers of $t$ in $R_i$. Hence, $R_i$ can be rewritten in terms of $t^{j}a_it^{-j}$, where $i=1,\ldots,n-1$ and $-\Delta_{\psi}(R_i)\leq j\leq \Delta_{\psi}(R_i)$. Since $P_1$ has a finite number of relators, then for all $i$, $i=1,\ldots,n+k$, $\Delta_{\psi}(R_i)$ is bounded by some $K\in\mathbb{N}$ and so $R_i$ for all $i$, $i=1,\ldots,n+k$, can be rewritten in terms of $t^{j}a_it^{-j}$, where $i=1,\ldots,n+k$ and $-K\leq j\leq K$. As all powers that appear in $P_1$ are divisible by a prime $p$, then the arguments in \cite{Stohr} carry through to conclude that $G_1$ is large.
\vskip 2mm

Now consider
\[
P_2=\langle a_1,\ldots,a_{n-1},t\mid R_1,\ldots,R_{n-2},R_{n-1}^{\alpha_1},\ldots,R_{n+k}^{\alpha_{k+2}},\ldots \rangle,
\]
where $\sigma_t(R_i)=0$ for all $i\in \mathbb{N}$, and $p$ divides $\alpha_j$, for all $j\in\mathbb{N}$. Once again, since $\sigma_t(R_i)=0$ for all $i$, then the group defined by $P_2$ surjects onto $\mathbb{Z}$ by sending $t$ to $1$ and the rest of the generators to $0$. If there is a $K\in\mathbb{N}$ such that $\Delta_{\psi}(R_i)\leq K$ for all $i\in\mathbb{N}$, $R_i$ can be rewritten in terms of $t^{j}a_lt^{-j}$ for all $i$, where $-K\leq j\leq K$ and $1\leq l\leq n-1$. As $p$ divides all the powers $\alpha_k$, then Stohr's arguments still carry through to conclude that $P_2$ defines a large group.
\vskip 2mm

As mentioned at the beginning of the section, we are interested in finding conditions under which a group $G$ with infinite abelianisation is large, regardless of its deficiency. In order to use Stohr's results, we first need to show that $G$ admits a presentation such as $P_1$ or $P_2$, where $\sigma_t(R_i)=0$ for all the relators present in the presentation. The following lemma shows that such a presentation always exists.
\vskip 2mm

\begin{lem}\label{mapZ}
Let $G$ be a group given by the presentation $P=\langle x_1,\ldots,x_d \mid u_1,\ldots,u_{s}\rangle$. Suppose $G$ admits a homomorphism $\phi$ onto the integers. Then $G$ admits a presentation $Q=\langle y_1,\ldots,y_d,t \mid r_1,\ldots,r_{s+1}\rangle$, such that $\phi\circ\varphi_Q$ maps $t$ to $1$ and $y_i$ to $0$ for all $i$, $1\leq i\leq d$, where $\varphi_Q$ is the canonical map from $F_{d+1}$ to $G$ induced by $Q$. In particular $\sigma_t(r_i)=0$ for all $i$, $1\leq i\leq s+1$.
%\begin{note}
%The deficiency of $Q$ is the same as the deficiency of $Q$
%\end{note}

\end{lem}
\begin{proof}
Consider the canonical map $\varphi_P$ from $F_d$ to $G$ induced by $P$. Fix $t$ an element of $F_d$ that maps to $1$ under $\phi\circ\varphi_P$. This element can be obtained as a word $\omega(x_1,\ldots,x_d)$ in terms of the generators $x_1,\ldots,x_d$. Consider $y_i:=x_it^{-\phi\circ\varphi_P(x_i)}$, for all $i$ such that $1\leq i\leq d$. Note that each element in the set $\{x_1,\ldots, x_d\}$ may be expressed in terms of elements in $\{y_1,\ldots, y_d,t\}$. Rewrite the relators $u_i$, $1\leq i\leq s$, and the word $wt^{-1}$ in terms of elements in $\{y_1,\ldots ,y_d,t\}$. Denote the rewritten relators by $r_i$, where $1\leq i\leq s+1$. Note that $r_i$ corresponds to the $u_i$ for $1\leq i\leq s$, while $r_{s+1}$ corresponds to $wt^{-1}$. Then, the presentation
\[
Q=\langle y_1,\ldots,y_d,t \mid r_1,\ldots,r_{s+1}\rangle
\]
is a presentation for $G$. 
\vskip 2mm

If $v$ is a reduced word in terms of elements in $\{y_1,\ldots, y_d,t\}$, then denote by $v'$ the word obtained from $v$ rewritten in terms of $\{x_1,\ldots, x_d\}$. Using this, define $\varphi_Q(v)$ to be $\varphi_P(v')$. Therefore $\phi\circ\varphi_Q(t)=1$ and $\phi\circ\varphi_Q(y_i)=0$ for all $i$, $1\leq i\leq d$. Moreover, given that $\phi\circ\varphi_Q(r_i)=0$ in $\mathbb{Z}$, then $\sigma_t(r_i)=0$, for all $i$, $1\leq i\leq s+1$.

\end{proof}
\vskip 2mm

\begin{rem}\label{rem: after lemma in inf abel case}
Let $u$ be a relator in the presentation $P$ from \cref{mapZ}. Suppose $u$ is a proper power relator, say $w^{n}=u$, with $n>1$. Denote by $r$ the relator in $Q$ corresponding to $u$ ($Q$ as in \cref{mapZ}). The relator $r$ is obtained by rewriting $u$ in terms of elements in $\{y_1,\ldots y_d,t\}$. This can be done by just rewriting $w$ and taking its $n$-th power. Call $w'$ the rewritten word of $w$ in terms of elements in $\{y_1,\ldots y_d,t\}$. Then $r=(w')^{n}$ and hence any $n$-th power relator from $P$ becomes an $n$-th power relator in $Q$.
\end{rem}

\vskip 2mm

\begin{rem}\label{rem: S}
Consider $P$, $G$ and $\phi$ as in \cref{mapZ}. Let $w\in F_d$ be a word written in terms of $x_1,\ldots,x_d$ and let $w'\in F_{d+1}$ be the word $w$ rewritten in terms of $y_1,\ldots,y_d,t$. Then, given how $y_1,\ldots,y_d,t$ are defined in terms of $x_1,\ldots,x_d$, $\Delta_{\phi}(w)=\Delta_{\phi}(w')$.
\end{rem}
\vskip 2mm

\begin{rem}
Note that the proof in \cref{mapZ} does not use the fact that there are a finite number of relators, we just add a generator and a relator, hence the statement holds for presentation of the type
\[
\langle x_1,\ldots,x_d \mid u_1,\ldots,u_{s},\ldots \rangle.
\]
\end{rem}
\vskip 2mm

In view of \cref{mapZ} and what we have discussed so far, we obtain
\vskip 2mm

\begin{thm}\label{thm: new G large with inf abel}
Let $G$ be defined by $P=\langle x_1,\ldots,x_n\mid w_1^{\alpha_1},\ldots,w_m^{\alpha_m},\ldots\rangle$, where $n>1$. Assume $\psi$ is a surjection from $G$ onto $\mathbb{Z}$ and suppose $\Delta_{\psi}(w_i)\leq K$, for all $i\in \mathbb{N}$, where $K$ is a fixed natural number. If at most $n-2$ powers $\alpha_m$ are not divided by a common prime number $p$, then $G$ is large.
\end{thm}
\begin{proof}
By \cref{mapZ}, $P_2$ is a presentation for $G$. By \cref{rem: S}, if $\Delta_{\psi}(w_i)\leq K$, then $\Delta_{\psi}(R_i)\leq K$ for all $i\in \mathbb{N}$. Therefore, $R_i$ for all $i\in \mathbb{N}$, can be rewritten in terms of $t^{j}a_it^{-j}$ where $-K\leq j\leq K$ and $i=1,\ldots,n-1$. By \cref{rem: after lemma in inf abel case}, each power $\alpha_m$ in $P$ becomes an power $\alpha_m$ in $P_2$. As at most $n-2$ of these powers do not have a common prime factor, then Stohr's arguments in \cite{Stohr} carry through to conclude $G$ is large.
\end{proof}
\vskip 2mm

\begin{cor}\label{thm: G large when inf abel and neg def}
Say $G$ is given by the presentation
\[
\langle x_1,\ldots,x_n \mid u^{m_1}_1,\ldots, u^{m_{s}}_{s}\rangle,
\]
where $n>1$. If $G$ has infinite abelianisation and at least $s-n+2$ relators %$u_{i_1}^{m_{i_1}},\ldots,u_{i_{s-d+2}}^{m_{i_{s-d+2}}}$
are such that their powers have a common factor, then $G$ is large.
\end{cor}
\begin{proof}
The condition that at least $s-n+2$ relators %$u_{i_1}^{m_{i_1}},\ldots,u_{i_{s-d+2}}^{m_{i_{s-d+2}}}$
are such that their powers have a common factor is equivalent to having at most $n-2$ relators with powers which are not divisible by a common prime factor. As there are only a finite number of relators, then for all $u_i$, $\Delta_{\psi}(u_i)$ is bounded, where $\psi$ is a surjective map from $G$ to $\mathbb{Z}$.
\end{proof}

\vskip 2mm

\begin{ex}
Consider $F_2=\langle a,t\rangle$, the non-abelian free group of rank $2$, freely generated by $a$ and $t$. Consider the sequence of commutators
\[
r_{j,1}=[a,t^{j}], r_{j,2}=\big[[a,t^{j}],a\big], r_{j,3}=[r_{j,1},r_{j,2}],\ldots,r_{j,m}=[r_{j,m-1},r_{j,m-2}],
\]
where $j\geq 1$. Let $G_j$ be defined by
\[
\langle a,t\mid r_{j,1}^{p},\ldots,r_{j,m}^{p},\ldots\rangle.
\]
As $r_{j,i}^{p}$ is in the commutator subgroup of $F_2$, $G_j$ surjects onto $\mathbb{Z}$. Moreover, $\Delta(r_{j,i}^{p})\leq j$. Therefore, by \cref{thm: new G large with inf abel} $G_j$ is large for all $j\in\mathbb{N}$. 
\vskip 2mm

We do not know if the groups $G_j$ are finitely presented. We suspect, however, that they are not.
\vskip 2mm

\end{ex}

\begin{ex}
It is very easy to construct examples for which \cref{thm: G large when inf abel and neg def} applies. Consider the presentation 
\[
P=\langle x_1,\ldots,x_n \mid u^{m_1}_1, \ldots, u^{m_{s}}_{s}\rangle,
\]
and impose the condition $\sigma_{x_{i}}(u_j)=0$, for some $x_i$ and all $u_j$, $1\leq j\leq s$. Then, $P$ defines a group with infinite abelianisation. Finally, only $s-n+2$ powers with a non-trivial common factor are needed to apply \cref{thm: G large when inf abel and neg def}.
\end{ex}

\end{document}